\documentclass{amsart}

\usepackage[T1]{fontenc}
\usepackage[utf8]{inputenc}
\usepackage{amsmath,amsthm,amsfonts,amscd,amssymb,eucal,latexsym,mathrsfs}
\usepackage{stmaryrd}
\usepackage{enumerate}
\usepackage{hyperref}
\usepackage[all]{xy}
\usepackage{etoolbox}

\usepackage{tikz,tikz-cd}
\usetikzlibrary{shapes.geometric}
\usetikzlibrary{arrows}
\usepackage{tqft}

\usetikzlibrary{positioning}
\usepackage{float}
\usepackage{MnSymbol}
\usetikzlibrary{matrix}

\makeatletter
\renewcommand\part{\@startsection {part}{1}{\z@}%
                                   {-3.5ex \@plus -1ex \@minus -.2ex}%
                                   {2.3ex \@plus.2ex}%
                                   {\newpage\centering\normalfont\bfseries}}
\makeatother

\newcommand{\toc}{\tableofcontents}

\theoremstyle{plain}
\newtheorem{theorem}{Theorem}[section]
\newtheorem*{theorem*}{Theorem}

\newtheorem{lemma}[theorem]{Lemma}

\theoremstyle{definition}
\newtheorem{remark}[theorem]{Remark}

\newtheorem{definition}[theorem]{Definition}

\usetikzlibrary{calc,graphs}
\usepackage{xcolor}
\usetikzlibrary{arrows,decorations.pathmorphing}

\hypersetup{colorlinks=false,linkbordercolor=black}

\newcommand{\p}{\varphi}

\newcommand{\e}{\varepsilon}

\newcommand{\IR}{\mathbb{R}}
\newcommand{\RI}{\mathbb{R}}

\newcommand{\ZI}{\mathbb{Z}}
\newcommand{\IN}{\mathbb{N}}
\newcommand{\NI}{\mathbb{N}}

\DeclareMathOperator{\inj}{\hookrightarrow}

\DeclareMathOperator{\lra}{\leftrightarrow}

\DeclareMathOperator{\Id}{\mathrm{Id}}

\newcommand{\IP}{\mathbb{P}}

\newcommand{\Bord}{\mathrm{Bord}}

\newcommand{\sq}{$\frac 7 4$}

\newcommand{\ts}{\textsection}

\newcommand{\ip}[1]{\langle#1\rangle} 

\author{Sylvain Barr\'e}
\author{Mika\"el Pichot}
\title{Random group cobordisms of rank \sq}
\begin{document}

\begin{abstract}
We construct a model of random groups of rank \sq, and show that in this model the random group has the exponential mesoscopic rank property.   
\end{abstract}

\maketitle

\section{Introduction}

This paper will define a model of random groups of rank \sq, and prove that in the given model the random group  has exponential mesoscopic rank. A group is said to be of rank \sq\ if it acts properly uniformly on a CAT(0) 2-complex all of whose vertex links are isomorphic to the Moebius--Kantor graph.

There are several models of random groups.  The standard constructions  start with a free (or hyperbolic) group and adjoin to it randomly chosen sets of relations. In general, the  random group  in these constructions is hyperbolic.  Gromov has proposed several  models for which this is the typical behavior.  We refer to \cite{Ollivier} for an introduction to the subject. 

Models of random groups of intermediate rank, in which the random group is generally non hyperbolic, were defined in \cite{random}. Roughly speaking, the random group in these models is defined by removing randomly chosen sets of relations in a given (non hyperbolic) group. It seems  difficult, however, to use this type of constructions to produce random groups of rank \sq.

The present paper takes a different approach and makes use of a category $\Bord_{\frac 7 4}$ of group cobordisms, introduced in \cite{surgery}, to define random groups of rank \sq. This is a kind of random surgery construction, where the random group is obtained by composing group cobordisms at random in the category $\Bord_{\frac 7 4}$. In particular, the fact that the random group is of rank \sq\ is tautological. 

This model relies on the existence of an infinite (finitely generated) semigroup of endomorphisms $E$  in the category $\Bord_{\frac 7 4}$, which we  introduce in \ts \ref{S - semigroup E} below (see Definition \ref{D - semigroup E}). The elements of $E$ are compositions of certain group cobordisms of rank \sq.
  To every element $\omega\in E$ is attached a group  $G_\omega$ of rank \sq. By random group of rank \sq, we  mean  the group $G_\omega$ associated with a randomly chosen element $\omega$ in the semigroup $E$.

The groups $G_\omega$ are in fact never hyperbolic. We have that:

\begin{theorem}\label{T - Z2 embeds}
The group $G_\omega$ contains $\ZI^2$ for all $\omega \in E$.
\end{theorem}

The proof shows that the surgery associated with the group cobordisms defining $G_\omega$ translates well into a surgery on tori in the associated complex. We  define a graph $R(\omega)$ for  $\omega\in E$ which is  shown to encode sufficiently many of these tori. This graph can be explicitly computed  using the fact that the cobordism compositions in $G_\omega$ correspond to simpler operations on $R(\omega)$.

Note that the above theorem holds for every $\omega$. We are interested in properties that only hold for ``most'' $\omega$. It is not hard to give a meaning to this expression (using the semigroup $E$) in our context (see \ts \ref{S - proof theorem 2}, Definition \ref{D - overwhelming proba}). Following standard terminology (introduced by Gromov), we  say that a property of $G_\omega$ holds  ``with overwhelming probability'' when the given condition holds. In the next two results, the definition of ``exponential rank'' and ``exponential mesoscopic rank'' are recalled in \ts\ref{S - proof theorem 2} and \ts\ref{S - proof theorem 3}
respectively.

\begin{theorem}\label{T - Random 74}
The group $G_\omega$ has exponential rank  with overwhelming probability.
\end{theorem}

In particular, there exist infinitely groups of rank \sq\ having exponential rank. This answers a question raised in \cite[\ts4.3]{rd}.

The proof of  Theorem \ref{T - Random 74} expands on that of Theorem \ref{T - Z2 embeds}, exploiting further properties of $R(\omega)$.
In view of this result, it is natural to ask if the random group in this model has  in fact the more demanding ``exponential mesoscopic rank'' property. This is indeed the case: 

\begin{theorem}\label{T - Random 74 meso}
The group $G_\omega$ has exponential mesoscopic rank with overwhelming probability.
\end{theorem}

While this can't be proved using $R(\omega)$ alone, we will see that $R(\omega)$ can be used to reduce the proof to that of the same property given in \cite[Theorem 11(b)]{rd}. The goal of \ts \ref{S - proof theorem 3} is to  explain  this reduction.

\bigskip 
\bigskip 
\bigskip

\noindent{Acknowledgments.} Part of this work was carried out during a visit of the second author to Universit\'e  Bretagne Sud partially funded by an NSERC discovery grant of the second author.

\toc

\section{The semigroup $E$}\label{S - semigroup E}

We start by constructing a 2-complex $C$ as follows. Consider six oriented equilateral triangles, with labeled edges, given by the following set of triples 
\[
(x,a,d),\ (y,c,d),\ (z,c,b),\ (x',d,a),\ (y',b,a),\ (z',b,c).
\]
The 2-complex $C$ is defined by identifying the edges of these triangles respecting the labels and the orientations. 

The complex $C$ arises naturally given the constructions in \cite{surgery}. It is a \emph{collar}, as defined in \cite[\ts 3]{surgery}, in a 2-complex of nonpositive curvature. Collars form object sets in categories of group cobordisms.  In the present case, using Lemma 9.1 in  \cite{surgery}, we see that:

\begin{lemma} 
The 2-complex $C$ is an object in the category $\Bord_{\frac 7 4}$.
\end{lemma}

The goal of this section is to define a semigroup $E$ of endomorphisms of $C$ in the category $\Bord_{\frac 7 4}$ of group cobordisms of rank \sq. We will quote the facts we need from \cite{surgery}.

Being a collar, the interior of $C$ is isomorphic to a product space of a segment with a graph called the nerve. The nerve of $C$ can be computed explicitly as follows:
 \begin{figure}[H]
 \begin{tikzpicture}[baseline=2.7ex]
 \coordinate (A) at (0,0);
 \coordinate (B) at (1,0);
 \coordinate (C) at (1,1);
 \coordinate (D) at (0,1);
    \draw (A) -- (B);
    \draw (C) -- (D);
    \draw (A) to [bend left] (D);
    \draw (B) to [bend left] (C);
    \draw (A) to [bend right] (D);
    \draw (B) to [bend right] (C);
    \draw (A) node[left] {$d$};
    \draw (B) node[right] {$c$};
    \draw (C) node[right] {$b$};
    \draw (D) node[left] {$a$};
  \end{tikzpicture}
 \caption{The nerve of $C$}
 \end{figure}
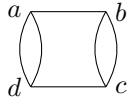
 
Lemma 6.3 in   \cite{surgery}  shows that there are exactly two minimal possibilities for the nerve  in a complex of nonpositive curvature. The graph $S$ shown above, and the graph $T$, which is the 1-skeleton of the tetrahedron. These minimal collars are said to be of type $ST$. The collar $C$ is a collar of type $S$.

The definition of the semigroup $E$ uses  two group cobordism classes of rank \sq\ defined as follows:

\begin{theorem}[{\cite[Theorem 10.1]{surgery}}] \label{T - cobordism 74}
There exist precisely two non filling group cobordisms of rank \sq\ with one vertex whose boundary collars are connected of type $ST$. Furthermore, the two cobordisms are self-dual, and their collars are pairwise isomorphic,  self-dual, and of type $S$.
\end{theorem}

We write $X$ and $Y$ for these two classes, which correspond, respectively, in the notation of \cite[\ts 10]{surgery}, to the rank $\frac 3 2$ cobordism, and the   rank $2$ cobordism. 
The 2-complex $C$ is coincides with the boundary collars in these two classes:

\begin{lemma}
The boundary collars in $X$ and $Y$ are  isomorphic to $C$. 
\end{lemma}

We will define several arrows---the generators of $E$---in the category $\Bord_{\frac 7 4}$ by combining the above theorem with the following lemma.

\begin{lemma}\label{L - involution}
The complex $C$ admits an automorphism of order 2, which is defined by the transformation $y\lra -y$, $y'\lra -y'$,  $x\lra -z$, $x'\lra -z'$, $a\lra -b$ and $c\lra -d$, where the notation ``$-y$'' indicates that the orientation of edge $y$ is reversed.
\end{lemma}
\begin{proof}
It is easy to check that the given map permutes the faces $(x,a,d)\lra (z,b,c)$, $(x',a,d)\lra (z',b,c)$, and that it preserves $(y,c,d)$ and $(y',b,a)$.
\end{proof}

Let us choose representatives for $X$ and $Y$ (again denoted $X$ and $Y$) in their isomorphism classes, together with identifications 
\[
L_X,R_X\colon C\inj X \text{ and }L_Y,R_Y\colon C\inj Y
\]
between $C$ and the left and right collar boundary components of $X$ and $Y$, respectively. 
The arrows in $\Bord_{\frac 7 4}$ will be denoted
$X_{ij}$ and $Y_{ij}$, for $i,j=0,1$. They correspond to prescribing the following injective boundary maps for $X$ and $Y$, respectively: 
\[
L_X\circ\sigma^{i} \colon C\inj X, \ \ \ R_X\circ\sigma^{j} \colon C\inj X
\]
for $X_{ij}$, and 
\[
L_Y\circ\sigma^{i} \colon C\inj Y, \ \ R_Y\circ\sigma^{j} \colon C\inj Y.
\]
for $Y_{ij}$, where $\sigma$ is the involution given by Lemma  \ref{L - involution}. 

We can now define the semigroup $E$ as follows.

\begin{definition}\label{D - semigroup E}
The  semigroup $E$ is the endomorphism semigroup of $C$ generated by the six arrows $X_{ij}$ and $Y_{ij}$ where $i,j=0,1$.
\end{definition}

We show in Lemma \ref{ L - generator equivalence 0} and Lemma \ref{ L - generator equivalence} below that the number of arrows is indeed six.  

Recall that two group cobordisms  $X\colon C\to D$ and $Y\colon C\to D$ are said to be equivalent if there is a 2-complex isomorphism $\p \colon X\to Y$, such that $\p\circ L_X = L_Y$ and $\p\circ R_X = R_Y$ (\cite[\ts 8]{surgery}). We write $X\sim Y$ for this equivalence relation. The arrows in the group cobordism category consists of equivalence classes of group cobordisms.

\begin{lemma}\label{ L - generator equivalence 0}
$Y_{00}\sim Y_{11}$ and $Y_{01}\sim Y_{10}$. 
\end{lemma}

\begin{proof}
It follows from the description of $Y$ in \cite[\ts 10]{surgery} that there exists an automorphism of $Y$ which induces the involution $\sigma$ on $L_X(C)$ and $R_X(C)$. In the notation of \cite[Fig.\ 4]{surgery}, this automorphism acts on the core of $Y$ by the permutation:
\[
(1,16)(2,15)(3,14)(4,13)(5,12)(6,11)(7,10)(8,9).
\]
By construction, this automorphism induces an equivalence between $Y_{00}$ and $Y_{11}$, and an equivalence between $Y_{01}$ and $Y_{10}$. 
\end{proof}

\begin{lemma}\label{ L - generator equivalence}
The cobordisms $X_{ij}$, for $i,j\in \{0,1\}$, $Y_{00}$, and $Y_{01}$, are pairwise non equivalent.  
\end{lemma}

\begin{proof}
By \cite[Theorem 10.1]{surgery}, we have $X_{ij}\not \sim Y_{kl}$ for $i,j,k,l\in \{0,1\}$. Let $Z$ refer to either $X$ or $Y$. If $Z_{ij}\sim Z_{ik}$, by an equivalence $\p\colon Z_{ij}\to Z_{ik}$, then the restriction of $\p$ to $L_{Z_{ij}}(C)$ coincides with $L_{Z}\circ \sigma^i\circ \sigma^{-i}\circ L_{Z}^{-1}=\Id$. Therefore, $\p$ restricts to the identity map from the core of $Z$ to itself. In particular, it induces an automorphism of $R_Z(C)$ which fixes its left topological boundary pointwise. Since the only such automorphism is the identity, it follows that $j=k$. The same argument implies that $Z_{ik}\not \sim Z_{ik}$. To prove the lemma, it remains to show that $X_{01}\not\sim X_{10}$. This follows from the fact that the only automorphism of the core of $X$ stabilizing the roots corresponding to the left and right collars is the identity (see \cite[Fig.\ 3]{surgery}).
\end{proof}

\begin{remark}
1) The semigroup $E$ is a subsemigroup of the  monoid of endomorphisms of $C$ generated by $X_{00}$, $Y_{00}$, and an additional element, which corresponds to $C$ where the left and right boundary morphisms are  given respectively by the identity and the order 2 automorphism of Lemma \ref{ L - generator equivalence 0}. This monoid is isomorphic to a free product of the integers $\IN$ with the direct product $\ZI/2\ZI\times \IN$. 

2) The cobordism $X_{00}$ and $Y_{00}$ can be described explicitly by following the proof of Theorem \ref{T - cobordism 74} in \cite[\ts 10]{surgery}. They are given by
\[
(x,a,d),(y,c,d), (z,c,b),\ (1,1,2),\ (2,a',d'),(4,c',d'),(3,c',b')
\]
\[
(4,d,a),(3,b,a), (2,b,c),\ (1,3,4),\ (x',d',a'),(y',b',a'),(z',b',c')
\]
and
\[
(x,a,d),(y,c,d), (z,c,b),\ (1,2,3),\ (4,a',d'),(2,c',d'),(1,c',b')
\]
\[
(1,d,a),(3,b,a), (4,b,c),\ (2,4,3),\ (x',d',a'),(y',b',a'),(z',b',c')
\]
where $L_{X_{00}},R_{X_{00}},L_{Y_{00}},R_{Y_{00}}$ are the obvious maps.
\end{remark}

Recall from \cite[\ts 10]{surgery} that a group of type $A$ is said to be \emph{accessible by surgery} if it is isomorphic to the fundamental group of quotient of the form $X/\sim$, where $X\colon C\to C$ is a cobordism of type $A$, $C$ is an object in $\Bord_A$, and $\sim$ identifies the two copies of $C$ in $X$. 

Let $\Lambda_A$ denote the set of pointed isomorphism classes of pointed complexes of type $A$. For the type \sq, we have a map $\omega\to X_\omega$
taking a word $\omega\in E$ in the generators $X_{ij}$ (for $i,j\in \{0,1\}$), $Y_{00}$, and $Y_{01}$ to the pointed topological space $X_\omega:=(X/\sim,*)$, which is an element in $\Lambda_{7/4}$, where $X\colon C\to C$ is a representative for the composition of cobordisms $w=w(X_{ij},Y_{00}, Y_{01})$. The associated fundamental group $G_\omega$ is a group of rank \sq:

\begin{definition}
 $G_\omega:= \pi_1(X_\omega)$.
 \end{definition}
 
By construction, these groups of rank \sq\ are accessible by surgery.

\section{Proof of Theorem \ref{T - Z2 embeds}}\label{S - theorem 1}

Let $\omega\in E$. We define a graph $R(\omega)$ as follows. The vertex set of $R(\omega)$ is the subset of the set of labels on the edges of $X_\omega$ which are attached to at least an edge of $R(\omega)$. An edge in $R(\omega)$ between two vertices, with  labels $a$ and $b$, say, is a pair of adjacent triangles two of whose opposite edges have identical labels, and  the other two have labels $a$ and $b$.

The graph $R(\omega)$ can be used to describe the effect of the surgery on the ``small cylinders'' inside the cobordism components, $X$ and $Y$, of $X_\omega$. The group surgeries in $X(\omega)$ descend to simpler graph surgeries in $R(\omega)$, and the concatenation of these cylinders can be read  on $R(\omega)$. For most $\omega\in E$, it can  be determined directly on the graph $R(\omega)$ itself that the associate complex $X_\omega/\sim$ contains tori.

For example, the two graphs $R(X_{00})$ and $R(Y_{00})$ are given by:

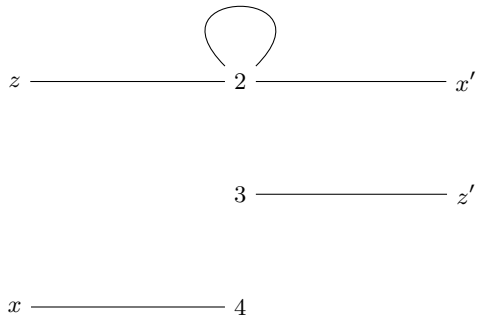
\begin{figure}[H]
\begin{tikzpicture}[scale = 3]
\tikzstyle{every node}=[font=\small]
\node (v202) at (1,.5){$z'$};
\node (v2) at (0,1){$2$};
\node (v4) at (0,0){$4$};
\node (v200) at (1,1){$x'$};
\node (v100) at (-1,0){$x$};
\node (v3) at (0,.5){$3$};
\node (v102) at (-1,1){$z$};
\draw[solid,thin,color=black,-](v2) to[loop]  (v2);
\draw[solid,thin,color=black,-] (v100) -- (v4);
\draw[solid,thin,color=black,-] (v102) -- (v2);
\draw[solid,thin,color=black,-] (v2) -- (v200);
\draw[solid,thin,color=black,-] (v3) -- (v202);
\end{tikzpicture}
\caption{The graph $R(X_{00})$}
\end{figure}

\begin{figure}[H]
\begin{tikzpicture}[scale = 3]
\tikzstyle{every node}=[font=\small]
\node (v100) at (-1,0){$x$};
\node (v200) at (1,1){$x'$};
\node (v202) at (1,0){$z'$};
\node (v4) at (0,1){$4$};
\node (v1) at (0,0){$1$};
\node (v102) at (-1,1){$z$};
\draw[solid,thin,color=black,-] (v1) -- (v4) ;
\draw[solid,thin,color=black,-] (v100) -- (v1) ;
\draw[solid,thin,color=black,-] (v102) -- (v4) ;
\draw[solid,thin,color=black,-] (v4) -- (v200) ;
\draw[solid,thin,color=black,-] (v1) -- (v202) ;
\end{tikzpicture}
\caption{The graph $R(Y_{00})$}
\end{figure}
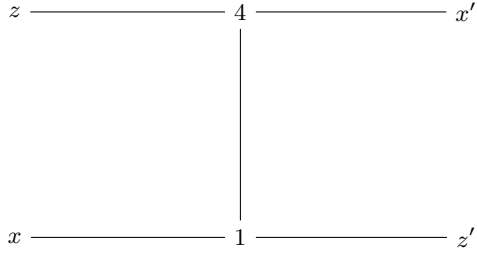

The surgery constructions used to build $X_\omega$  correspond to decompositions of the graph $R(\omega)$ in terms of the two graphs $R(X_{00})$, $R(Y_{00})$, taking into account the effect of the twist.

\begin{lemma}\label{L - subword}
If $\omega'$ is a subword of $\omega$ then $R({\omega'})$ embeds in $R(\omega)$. 
 \end{lemma}

This follows directly from the definition of $R(\omega)$. 
The following lemma describes the effect of the flip.

\begin{lemma}\label{L - flip}
Let $Z,T\in \{X,Y\}$ and $i,j,k,l\in \{0,1\}$. Then $R(Z_{ij}T_{kl})$ is obtained from the graph $R(Z_{i\bar j}T_{kl})$ by a surgery that flips the subgraph $R(Z_{ij})$ horizontally.  
\end{lemma}

\begin{proof}[Proof of Lemma \ref{L - flip}]
By Lemma \ref{L - involution}, the automorphism $\sigma$ of $C$ exchanges $x\lra z$ and $x'\lra z'$.
\end{proof}

\begin{lemma}\label{L - cyclic permutations}
Let $\omega,\omega'\in E$ and assume that $\omega'$ is a cyclic permutation of $\omega$. Then  
$G_{\omega'}\simeq G_{\omega}$.
\end{lemma}

Indeed, we have in fact an isomorphism at the level of complexes $X_{\omega'}\simeq X_\omega$.

\begin{lemma}\label{L - cycle Z2}
If $R(\omega)$ contains a (non necessarily simple) cycle $\gamma$ then at least one of the following two assertions holds:
\begin{enumerate}
\item $\gamma$ is supported on a loop of $R(\omega)$
\item $\ZI^2$ embeds in $G_\omega$. 
\end{enumerate}
\end{lemma}

\begin{proof}[Proof of Lemma \ref{L - cycle Z2}]
Let $\gamma$ be a cycle of $R(\omega)$. By definition of $R(\omega)$, two consecutive edges of $\gamma$ correspond to two consecutive cylinders in $X_\omega$. If (1) fails then $\gamma$ visits are least two distinct edges of $R(\omega)$. Let us reduce  $\gamma$ to a new cycle (still denoted $\gamma$) with the property that any two consecutive edges are distinct (i.e., they have distinct labels). In that case, the two tori cylinders associated with these edges are also distinct in $X_\omega$. Since the length of $\gamma$ is at least 2, this gives a (not necessarily embedded) torus in $X_\omega$, which in turns provides a copy of $\ZI^2$ in $G_\omega = \pi_1(X_\omega/\sim)$.    
\end{proof}

Let us assume first that the word $\omega$ is a power of $X_{00}$, or of $X_{11}$. In that case it is easily seen that $X_\omega/\sim$ is a finite normal cover of the complex $V_1$ defined in \cite{rd} (which is therefore accessible by surgery). As we checked in \cite{rd}, $\pi_1(V_1)$ contains $\ZI^2$, and therefore, $G_\omega$ contains $\ZI^2$ (explicitly, if $a_1,\ldots, a_8$ denotes the generators of the presentation given in \cite[\ts 4]{rd}, then $a_7$ commutes to $a_4a_5$).
The same argument works if $\omega$ is a power of $Y_{00}$, and it is not hard to check that the corresponding complex in \cite{rd} is $V_0^1$.
If $\omega$ is not a power of $X_{00}$, $X_{11}$, or of $Y_{00}$, then by Lemma \ref{ L - generator equivalence 0}, it contains either $X_{01}X_{00}$, $X_{00}Y_{00}$, or $Y_{00}Y_{00}$, as a subword, up to isomorphism of the corresponding graphs $R$. For instance, $R(X_{00}Y_{00})$ is a reflection of $R(Y_{00}X_{00})$, and $R(X_{00}Y_{00})$ and $R(X_{00}Y_{01})$ coincide up to to switching the right strands, by Lemma \ref{L - flip}.

The graph $R(X_{01}X_{00})$ is given by:

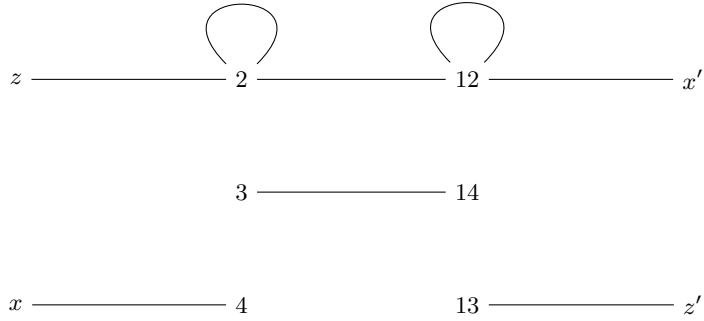
\begin{figure}[H]
\begin{tikzpicture}[shift ={(0.0,0.0)},scale = 3.0]
\tikzstyle{every node}=[font=\small]

\node (v12) at (2,1){$12$};
\node (v202) at (3,0){$z'$};
\node (v14) at (2,.5){$14$};
\node (v2) at (1,1){$2$};
\node (v4) at (1,0){$4$};
\node (v200) at (3,1){$x'$};
\node (v100) at (0,0){$x$};
\node (v102) at (0,1){$z$};
\node (v3) at (1,.5){$3$};
\node (v13) at (2,0){$13$};
\draw[solid,thin,color=black,-](v2) to[loop] (v2);
\draw[solid,thin,color=black,-] (v100) -- (v4);
\draw[solid,thin,color=black,-] (v102) -- (v2);
\draw[solid,thin,color=black,-](v12) to[loop](v12);
\draw[solid,thin,color=black,-] (v3) -- (v14);
\draw[solid,thin,color=black,-] (v2) -- (v12);
\draw[solid,thin,color=black,-] (v12) -- (v200);
\draw[solid,thin,color=black,-] (v13) -- (v202);

\end{tikzpicture}  
\caption{The graph $R(X_{01}X_{00})$}\label{Fig - X01X00}
\end{figure}

Therefore, $R(X_{01}X_{00})$ contains a cycle whose support is not reduced to  a loop. By Lemma \ref{L - subword}, if $\omega$ contains $X_{01}X_{00}$ then $R(X_{01}X_{00})$ embeds in $R(\omega)$, and Lemma \ref{L - cycle Z2} shows that $\ZI^2$ embeds in $G_\omega$. 

From now on we assume that $\omega$ contains either $X_{00}Y_{00}$ or $X_{00}Y_{01}$ as a subword, but not $X_{01}X_{00}$. The graph  $R(X_{00}Y_{00})$ can be represented as follows.

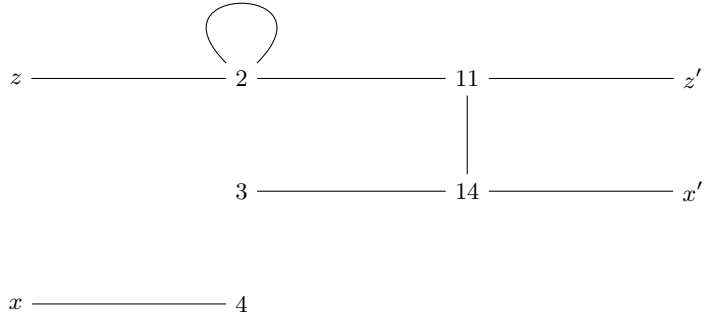
\begin{figure}[H]
\begin{tikzpicture}[shift ={(0.0,0.0)},scale = 3.0]
\tikzstyle{every node}=[font=\small]

\node (v202) at (3.0,1){$z'$};
\node (v14) at (2,.5){$14$};
\node (v2) at (1,1){$2$};
\node (v4) at (1,0){$4$};
\node (v200) at (3,0.5){$x'$};
\node (v100) at (0,0){$x$};
\node (v102) at (0,1){$z$};
\node (v11) at (2,1){$11$};
\node (v3) at (1,.5){$3$};
\draw[solid,thin,color=black,-](v2) to[loop](v2);
\draw[solid,thin,color=black,-] (v100) -- (v4);
\draw[solid,thin,color=black,-] (v102) -- (v2);
\draw[solid,thin,color=black,-] (v11) -- (v14);
\draw[solid,thin,color=black,-] (v2) -- (v11);
\draw[solid,thin,color=black,-] (v3) -- (v14);
\draw[solid,thin,color=black,-] (v14) -- (v200);
\draw[solid,thin,color=black,-] (v11) -- (v202);

\end{tikzpicture}  
\caption{The graph $R(X_{00}Y_{00})$}
\end{figure}

The word $\omega$ is either of the form $(X_{00})^kY_{00}$ for some $k\geq 1$, or there exists a letter $Z$ such that either 
$X_{00}Y_{00}Z$ or $X_{00}Y_{01}Z$ appears as a subword  of $\omega$. If $\omega'$ denotes the latter subword, then the graph $R(\omega')$ contains a cycle which is not a loop, so $G_\omega$ contains $\ZI^2$ in the second case. In the first case, Lemma \ref{L - cyclic permutations} applies to show that   $G_\omega$ contains $\ZI^2$ if $k\geq 2$. Finally, it is easy to check that $\ZI^2\inj G_\omega$ if $\omega=X_{00}Y_{00}$. 

This concludes the proof of Theorem  \ref{T - Z2 embeds}.

\section{Proof of Theorem \ref{T - Random 74}}\label{S - proof theorem 2}

The proof of Theorem \ref{T - Z2 embeds} suggests that stronger results should hold, perhaps only requiring that they  hold for ``almost all'' groups $G_\omega$. 

Let $n$ be an integer. Let $E_n$ denote the sphere of radius $n$ in the semigroup $E$, with respect to the generating set consisting of $X_{ij},  Y_{0k}$, for  $i,j,k\in \{0,1\}$. Let $\IP_n$ denote the uniform probability measure on $E_n$. 

In the present section we consider groups of the form $G_\omega$ where the parameter $\omega$ is chosen at random in $E_n$, with respect to $\IP_n$, for a large enough integer $n$.

\begin{definition}\label{D - overwhelming proba}  
If $P$ is a group property and 
\[
\IP_n\{\omega\in S_n,\ G_\omega\text{ has property } P\}
\]
converges to 1, we say that the random group of rank \sq\ has property $P$ with overwhelming probability (with respect to $(\IP_n)$). 
\end{definition}

Let us now recall the notion of polynomial and exponential rank growth from \cite[Definition 2]{rd}. As in \cite{rd} we shall refer to CAT(0) 2-complex built from equilateral triangles simply as triangle complexes.  
 
\begin{definition}\label{def2}
A triangle complex $X$ is said to have \emph{polynomial rank growth} (resp.\ \emph{subexponential rank growth}) if there exists a polynomial $p$ (resp.\ a function $p$ such that $\lim_r p(r)^{1/r}=1$) such that for any simplicial geodesic segment $\gamma$ in $X$, the number of flat equilateral triangles in $X$ with base $\gamma$ is bounded by $P(r)$, where $r$ is the length of $\gamma$. 
\end{definition}

A triangle group is said to have polynomial rank growth (resp. subexponential rank growth) if it admits a proper and isometric action on a triangle complex of polynomial rank (resp.\ subexponential rank growth), and it is said to be of exponential rank when it is not of subexponential rank growth.
We shall abbreviate `exponential rank growth' to `exponential rank'. 

Our goal in this section is to prove the following result.

\begin{theorem}\label{T - exponential rank}
The group $G_\omega$ has exponential  rank with overwhelming probability.
\end{theorem}

\begin{lemma}\label{L - cycle Z2 - pairs}
Assume that $R(\omega)$ contains two (non necessarily simple) cycles $\gamma$ and $\gamma'$ such that
\begin{enumerate} 
\item[a)] neither $\gamma$ nor $\gamma'$
 is supported on a loop of $R(\omega)$
\item[b)] $\gamma$ and $\gamma'$ are distinct and have a non-empty intersection
\end{enumerate}
Then   $G_\omega$ has exponential rank. 
\end{lemma}

\begin{proof}  
Consider $\gamma$ and $\gamma'$ satisfying the assumptions and let us prove b). Since $\gamma\cap \gamma'$ is non empty, there exists a maximal (non necessarily simple) segment $s=[x,y]$ in the intersection. Let us consider two segments  $\gamma_x$ (resp.\ $\gamma_y$)  in $\gamma$, defined to be the smallest segment with origin $x$ (respectively $y$) in the direction opposite to $[x,y]$ whose extremity belongs to $\{x,y\}$. We also define similarly two segments $\gamma_x'$ and $\gamma_y'$ in $\gamma'$.   In the case that $\gamma$ is a simple cycle, for example, then $\gamma_x=-\gamma_y$.   The segments $s,\gamma_x,\gamma_y,\gamma_x',\gamma_y'$ in the graph $R(\omega)$ correspond respectively to cylinders  in the complex $X_\omega$. In the universal cover $\tilde X_\omega$, these cylinders give rise to flat strips of uniformly bounded height (by the maximal height of the cylinders). In turn, these flat strips provide an isometric embedding $T\times \IR \inj \tilde X_\omega$, where $T$ is a metric binary tree with uniformly bounded edge lengths.  This clearly implies that $\tilde X$  has exponential rank.
\end{proof}

\begin{lemma}\label{L - forbidden cobordisms}
Assume that $|\omega|\geq 3$. The assumptions of Lemma \ref{L - cycle Z2 - pairs} as satisfied if and only if the word $\omega$ contains either one of the following
\begin{enumerate}[(a)]
\item $Y_{0*}Y_{0*}$
\item $X_{*k}X_{\bar kl}X_{\bar l*}$ 
\item $X_{**}Y_{0*}X_{**}Y_{0*}X_{**}$.
\item $X_{**}Y_{0*}X_{*k}X_{\bar k*}Y_{0*}X_{**}$.
\end{enumerate}
where $*,k,l\in\{0,1\}$.
\end{lemma}

\begin{proof}
By the results of the previous section, it is easy to check that the assumptions of Lemma \ref{L - cycle Z2 - pairs}  are satisfied if $\omega$ contains of the prescribed cobordisms, using that the fact that $|\omega'|\geq 3$ in case (a). (For instance, if $|\omega'|=3$, then one can construct two cycles starting from the additional letter $Z$, moving either to the left or to the right in the corresponding subgraph $R(Z)$.)  

Conversely, let $\gamma$ and $\gamma'$ be two cycles of minimal length satisfying the assumptions of the lemma, and let $\omega'\subset \omega$ be the minimal subword of $\omega$ such that $R(\omega')$ contains $\gamma\cup \gamma'$.  Note that if $|\omega'|=3$, then either $\omega'$ contains $Y_{**}Y_{**}$ as a subword, $\omega'=X_{*k}X_{\bar kl}X_{\bar l*}$. Otherwise, $|\omega'|\geq 4$ and $Y_{**}Y_{**}$ is not a subword of $\omega$, which must therefore contains  $X_{**}Y_{**}X_{**}$ every time it contains $Y_{**}$.  The result then follows easily from the next lemma, since by minimality at least one of the left strands in $R(\omega')$ is connected to one of the right strands.
\end{proof}

\begin{lemma}
The left strands in the graph $R(X_{*k}X_{k*})$ are disconnected from the right strands.   
\end{lemma}
\begin{proof}
  By Lemma \ref{L - flip}, the graph $R(X_{*k}X_{k*})$ is obtained from $R(X_{*0}X_{1*})$ by flipping the central strands, whose effect is to disconnect the left hand side from the right hand side (see Figure \ref{Fig - X01X00}). 
\end{proof}

The growth of the semi-group $E$ is defined as follows
\[
g:=\lim_{n\to \infty} {|E_n|}^{1/n},
\]
where $E_n$ denotes the sphere of radius $n$ in the monoid $E$. We refer to \cite[Chap.\ VI]{Harpe} for an introduction to the notion of growth.
  
We shall compare $g$ to a slight variation defined as follows. Let $E'_n$ denote the set of cobordisms of length $n$ in $E$ which do not contain $Y_{0*}Y_{0*}$. This is an exponentially growing number of cobordisms, of  which we have to find an upper bound. 
The relevant modification of $g$, in view of the above lemma, is:
\[
g':=\lim_{n\to \infty} {|E_n'|}^{1/n}.
\]
As for the usual growth, the limit converges by submultiplicativity (compare \cite[p.\ 154]{Harpe}.)

\begin{lemma}\label{L - gleqg} $g'<g$. \end{lemma}
\begin{proof} 
Note that a word ending with $X_{\e0}$ or $X_{\e1}$ can lead to 6 possible new cobordisms, which end either with $X_{\e*}X_{**}$ (4 new cobordisms) or with $X_{\e*}Y_{**}$ (2 new cobordisms). The same holds for cobordisms ending with $Y_{00}$ or $Y_{01}$. Thus, $E_{n+1} = 3E_n$ and $g=3$.

Let $E_n''$ denote the set of cobordisms of length $n$ in $E$ which end with $Y_{00}Y_{00}$ but do not contain another copy of  $Y_{00}Y_{00}$. Thus, for instance, $|E_1'|=6$, $|E_2'|=17$, $|E_2''|=1$ and $|E_3''|=5$.

If one adds a letter to a word in $E_n'$ then one obtains either a word in $E_{n+1}'$ or  one in $E_{n+1}''$. Since adding a letter creates three times as many cobordisms, this gives the relation:
\[
|E_{n+1}'|+|E_{n+1}''| = 3|E_n'|.
\]
In order to find another relation, let us add the letters $Y_{00}Y_{00}$ to a word in $E_n'$. This results either in an element of $E_{n+2}''$, or in a word in $E_{n+2}$ which ends with  $Y_{00}Y_{00}Y_{00}$. The corresponding relation can be written as follows:
\[
|E_{n+2}''| +|E_{n+1}''|=|E_n'|.
\]
Solving the order 2 recurrence relation, the roots are
\[
{\sqrt 3}-1 \text{ and }
g'=1+{\sqrt 3}<2.74<g.
\]
\end{proof}

\begin{proof}[Proof of Theorem \ref{T - exponential rank}]
If $\Omega$ denotes the set of words in Lemma \ref{L - forbidden cobordisms}, we have for $n\geq 3$
\begin{align*}
\IP_n(G_\omega\text{ has exponential rank})&\geq \IP_n(\omega\text{ contains a subword in }\Omega)\\
&\geq \IP_n(\omega\text{ contains }Y_{00}Y_{00})\\ 
&= 1 - \IP_n(\omega\text{ does not contain }Y_{00}Y_{00})\\ 
&= 1 - \frac{|E'_n|}{|E_n|}\\ 
&\sim 1 - \left (\frac{g'}{g}\right)^n
\end{align*}
which converges  to 1 exponentially fast.
\end{proof}

\section{Proof of Theorem \ref{T - Random 74 meso}}\label{S - proof theorem 3}

Theorem \ref{T - Random 74 meso} concerns mesoscopic rank, which is defined as follows. 
Let $X$ be a CAT(0) 2-complex and $A$ be a point of $X$ and consider function 
\[
\varphi_A : \RI_+ \to \NI
\] 
which takes $r\in \RI_+$ to the number of flat disks in $X$ of center $A$  which are not included in a flat plane in $X$.

\begin{definition}[See {\cite{rd}}] The function
$\varphi_A$ is called the \emph{mesoscopic profile} at $A$. The 2-complex $X$ is said has \emph{exponential mesoscopic rank} (at $A$) if the function $\p_A$ grows exponentially fast.
\end{definition}

We say that a group has exponential mesoscopic rank if it acts properly uniformly on such a complex having exponential mesoscopic rank at least at one point. Our goal in this section is to prove the following result.

\begin{theorem}
The group $G_\omega$ has exponential mesoscopic  rank with overwhelming probability.
\end{theorem}

Consider the group cobordism $\omega_0 := Y_{00}Y_{00}Y_{00}$. The proof of the theorem relies on the following lemma.

\begin{lemma}\label{L - meso}
If a word $\omega$ contains $\omega_0$ as a subword, then $G_{\omega}$ has exponential mesoscopic rank.
\end{lemma}

Note that the lemma implies the theorem, by an argument similar to that of the previous section. namely, if $E'_n$ denote the set of cobordisms of length $n$ in $E$ which do not contain $\omega_0$, and
\[
g':=\lim_{n\to \infty} {|E_n'|}^{1/n}
\]
then an analog of Lemma \ref{L - gleqg} shows that  
\[
g'<g
\]
and therefore
\begin{align*}
\IP_n(G_\omega\text{ has exponential mesoscopic rank})&\geq \IP_n(\omega\text{ contains }\omega_0)\\ 
&= 1 - \IP_n(\omega\text{ does not contain }\omega_0)\\ 
&= 1 - \frac{|E'_n|}{|E_n|}\\ 
&\sim 1 - \left (\frac{g'}{g}\right)^n
\end{align*}
which converges  to 1 exponentially fast.

\begin{proof}[Proof of Lemma \ref{L - meso}] Our goal is to show that Lemma 57 in \cite{rd}, which establishes exponential mesoscopic rank, can be applied in the present situation. 

The following presentation of $\omega_0$ is obtained by pasting together three times the  presentation of $Y_{00}$ given in \ts \ref{S - semigroup E}:
\begin{align*}
&(x,5,6),(y,7,6),(z,7,8),(1,2,3),\   (4,15,16),(2,17,16),(1,17,18),(11,12,13),\\
&\ \ \ \ \ \ \ \ \ (14,25,26),(12,27,26),(11,27,28),(21,26,25), (24,35,36),(22,37,36),(21,37,38)
\end{align*}
\begin{align*}
&(1,6,5),(3,8,5),(4,8,7),(2,4,3),\ (11,16,15),(13,18,15),(14,18,17),(12,14,13),\\
&\ \ \ \ \ \ \ \ \ (23,28,25),(24,28,27),(21,22,23),(22,24,23), (x',36,35),(y',38,35),(z',38,37)
\end{align*}

The graph $R(\omega_0)$ (see \ts\ref{S - theorem 1})  is obtained by gluing  together three copies of the graph $R(Y_{00})$. It is given by

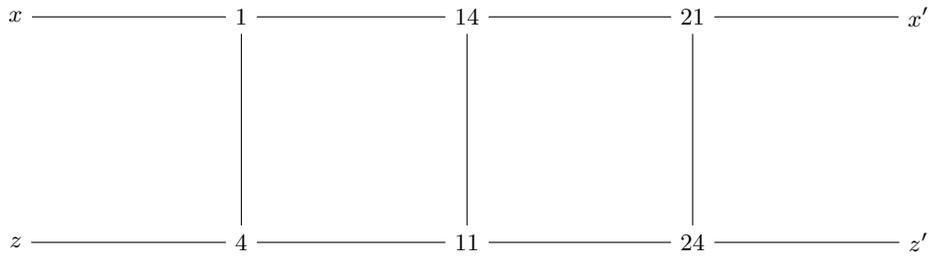
\begin{tikzpicture}[shift ={(0.0,0.0)},scale = 3.0]
\tikzstyle{every node}=[font=\small]

\node (v202) at (4,1){$x'$};
\node (v14) at (2,1){$14$};
\node (v21) at (3,1){$21$};
\node (v4) at (1,0){$4$};
\node (v200) at (4,0){$z'$};
\node (v100) at (0,1){$x$};
\node (v24) at (3,0){$24$};
\node (v102) at (0,0){$z$};
\node (v11) at (2,0){$11$};
\node (v1) at (1,1){$1$};
\draw[solid,thin,color=black,-] (v1) -- (v4);
\draw[solid,thin,color=black,-] (v100) -- (v1);
\draw[solid,thin,color=black,-] (v102) -- (v4);
\draw[solid,thin,color=black,-] (v11) -- (v14);
\draw[solid,thin,color=black,-] (v4) -- (v11);
\draw[solid,thin,color=black,-] (v1) -- (v14);
\draw[solid,thin,color=black,-] (v21) -- (v24);
\draw[solid,thin,color=black,-] (v14) -- (v21);
\draw[solid,thin,color=black,-] (v11) -- (v24);
\draw[solid,thin,color=black,-] (v24) -- (v200);
\draw[solid,thin,color=black,-] (v21) -- (v202);

\end{tikzpicture} 

In order to exhibit exponential mesoscopic rank, consider the outer cycle 
\[
1-14-21-24-11-4
\] 
in this graph. It  generates a flat plane 
\[
\Pi=\ip{\gamma,\gamma'}
\]
supported by two transverse geodesics $\gamma,\gamma'\subset \Pi$. A direct computation shows 
 $\gamma = \ip A$, where 
\[
A = 18\cdot 26 \cdot 23\cdot 27\cdot 15 \cdot 2 
\] 
and $\gamma'=\ip B$, where $B= 14$, are two such geodesics. 

We have to show that this leads to exponential mesoscopic rank. We follow the strategy in \cite[\ts6.1]{rd}.  

 The argument in Lemma  \ref{L - cycle Z2 - pairs} applied to $R(\omega_0)$ shows that the embeddings $\gamma,\gamma'\inj \Pi$ extends to an embedding  $\gamma\times T \inj \tilde X_{\omega_0}$ where $\ip{\gamma, \gamma'}= \Pi$ and $T$ is a tree of exponential growth with uniformly bounded edges.    

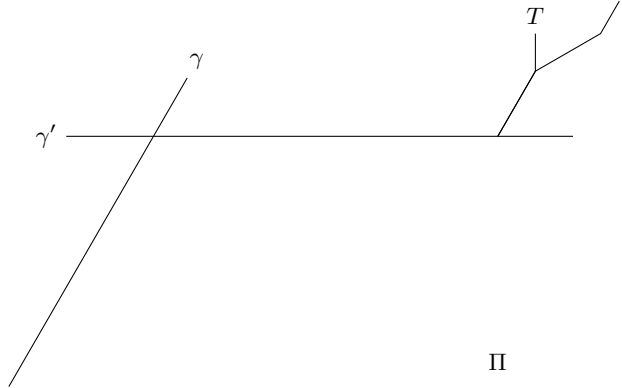
\begin{center}
\begin{tikzpicture}[shift ={(0.0,0.0)},scale = 1.0]
\tikzstyle{every node}=[font=\small]

\node (A) at (-1,0){$\gamma'$};
\node (B) at (1,1){$\gamma$};
\node (P) at (5,-3){$\Pi$};

\draw (A) -- ++(0:7cm);
\draw (B) -- ++(-120:5cm);
\draw (5,0) -- ++(60:1cm) -- ++(30:1cm) -- ++(60:.5cm); 
\draw (5,0) -- ++(60:1cm) -- ++(90:.5cm) node[above]{$T$};

\end{tikzpicture}
\end{center}

The presentation of ${\omega_0}$ provides  a geodesic parallel to $\gamma$, namely $\gamma_1=\ip {A_1}$,  where $A_1= 14\cdot 11$, which corresponds to a strip of height 1 and period 6 on $\gamma$. This strip can thus be extended into a flat containing $\gamma$, whose intersection with $\Pi$ is a half-plane.

The proof of Lemma 57 in \cite{rd} then applies verbatim, with different constants,  and this implies that $G_\omega$ has exponential mesoscopic rank  by a local analysis. We will omit the details. 
\end{proof}

\end{document}